\documentclass[leqno,12pt]{amsart}
\usepackage{amsthm}
\usepackage{amssymb}
\usepackage{verbatim}
\usepackage{enumerate}
\usepackage{needspace}

\numberwithin{equation}{section}

\theoremstyle{plain}
\newtheorem{thm}[equation]{Theorem}

\newtheorem{cor}[equation]{Corollary}
\newtheorem{lemma}[equation]{Lemma}

\theoremstyle{definition}

\newtheorem{remark}[equation]{Remark}

\newcommand{\F}{\mathbb F}

\title[Permutation polynomials]{Permutation polynomials induced from permutations of subfields, and some
complete sets of mutually orthogonal latin squares}

\author{Michael E. Zieve}
\address{
  Department of Mathematics,
  University of Michigan,
  Ann Arbor, MI 48109--1043,
  USA
}
\address{Mathematical Sciences Center, Tsinghua University, Beijing 100084, China}
\email{zieve@umich.edu}
\urladdr{www.math.lsa.umich.edu/$\sim$zieve/}

\thanks{The author thanks Xiwang Cao and Baofeng Wu for comments on an earlier version of this paper, and the NSF for support under grant DMS-1162181.}

\begin{document}

\begin{abstract}
We present a general technique for obtaining permutation polynomials over a finite field from permutations of a
subfield.  By applying this technique to the simplest classes of permutation polynomials on the subfield, we obtain several new
families of permutation polynomials.  Some of these have the additional property that both
$f(x)$ and $f(x)+x$ induce permutations of the field, which has combinatorial consequences.  We use some of our permutation
polynomials to exhibit complete sets of mutually orthogonal latin squares.  In addition, we solve the open problem from a
recent paper by Wu and Lin, and we give simpler proofs of much more general
versions of the results in two other recent papers.
\end{abstract}

\maketitle

\section{Introduction}

A \textit{complete mapping} of a group $G$ is a permutation $\phi$ of $G$ for which the map
$\alpha\mapsto\alpha\phi(\alpha)$ is also a permutation of $G$.
Complete mappings were introduced by Mann \cite{Mann}, in the context of constructing orthogonal latin squares.
Complete mappings were subsequently shown to have several other combinatorial consequences,
including neofields, Bol loops, and partitions of $G$ for which the partial products have certain properties
\cite{FGM,FGT,NR,Paige}.
In the past few decades, several authors have studied complete mappings on the additive group of a finite field $\F_q$.
Since every function $\F_q\to\F_q$ can be written as $\alpha\mapsto f(\alpha)$ for some polynomial $f(x)\in\F_q[x]$, 
complete mappings on $\F_q$ can be studied in terms of polynomials, which facilitates the use of algebraic techniques.

We use the following terminology: a \emph{permutation polynomial} over $\F_q$ is a polynomial $f(x)\in\F_q[x]$ for which the function
$\alpha\mapsto f(\alpha)$ defines a permutation of $\F_q$.  A \emph{complete permutation polynomial} over $\F_q$
is a polynomial $f(x)\in\F_q[x]$ for which the function $\alpha\mapsto f(\alpha)$ defines a complete mapping on $\F_q$; in other words, both $f(x)$ and $f(x)+x$ are permutation polynomials over $\F_q$.
Recently several papers have been written exhibiting families of permutation polynomials, some of which are complete.
Many of the results in these papers are easy consequences of a general result (Lemma~\ref{lem}) which reduces the question
of whether a certain type of polynomial permutes $\F_q$ to the question of whether a different polynomial permutes some
subgroup of $\F_q^*$.

In this note we show that this method applies especially well in case the subgroup of $\F_q^*$ is the multiplicative group
of a subfield $\F_Q$ of $\F_q$.  In this case we can use Lemma~\ref{lem} to produce new permutation polynomials over $\F_q$
from known permutation polynomials over $\F_Q$.  Our construction is based on the following result:

\begin{thm} \label{main}
Pick $h\in\F_q[x]$ where $q=Q^m$, and let $r$ be a positive integer.  Then $x^r h(x^{(q-1)/(Q-1)})$ permutes\/ $\F_q$
if and only if
\begin{enumerate}
\item $\gcd(r,(q-1)/(Q-1))=1$ \,\text{ and}
\item $g(x):=x^r h(x) h^{(Q)}(x) h^{(Q^2)}(x) \dots h^{(Q^{m-1})}(x)$ permutes\/ $\F_Q$,
\end{enumerate}
where $h^{(Q^i)}(x)$ denotes the polynomial obtained from $h(x)$ by raising every coefficient to the $Q^i$-th power.
\end{thm}

We emphasize that the usefulness of this result is that permutation polynomials over $\F_{Q^m}$ are being constructed from
permutation polynomials over $\F_Q$.  This enables one to start with well-understood permutation
polynomials over $\F_Q$, such as $x^n$ or Dickson polynomials, and use them to construct new permutation polynomials
over $\F_{Q^m}$.  We also remark that $g(x)=x^r N(h(x))$, where $N$ denotes the norm relative to the field extension
$\F_q(x)/\F_Q(x)$; in particular, it follows that $g(x)\in\F_Q[x]$.

Theorem~\ref{main} becomes especially simple in
case $h(x)\in\F_Q[x]$:

\begin{cor} \label{maincor}
Pick any $h\in\F_Q[x]$, let $q=Q^m$, and let $r$ be a positive integer.  Then $x^r h(x^{(q-1)/(Q-1)})$ permutes\/ $\F_q$
if and only if
\begin{enumerate}
\item $\gcd(r,(q-1)/(Q-1))=1$ \,\text{ and}
\item $g(x):=x^r h(x)^m$ permutes\/ $\F_Q$.
\end{enumerate}
\end{cor}

Condition (2) can be simplified when $m$ is coprime to $Q-1$:

\begin{cor} \label{maincor2}
Pick any $h\in\F_Q[x]$, let $r,m,n$ be positive integers such that $mn\equiv 1\pmod{Q-1}$, and put $q=Q^m$.  Then $x^r h(x^{(q-1)/(Q-1)})$ permutes\/ $\F_q$
if and only if
\begin{enumerate}
\item $\gcd(r,(q-1)/(Q-1))=1$ \,\text{ and}
\item $g(x):=x^{rn} h(x)$ permutes\/ $\F_Q$.
\end{enumerate}
\end{cor}

Note that in this corollary we can begin with an arbitrary permutation polynomial $H(x)$ over $\F_Q$; writing
$H(x)-H(0)=x h(x)$, it follows that $x h(x^{(q-1)/(Q-1)})$ permutes $\F_q$ whenever $q=Q^m$ with $m\equiv 1\pmod{Q-1}$.
Thus, this special case of our construction enables us to produce permutation polynomials over extensions of $\F_Q$ from
an arbitrary permutation polynomial over $\F_Q$.  These permutation polynomials have coefficients in $\F_Q$, 
which is a situation studied by Carlitz and Hayes~\cite{CH}; however, our approach is different than theirs, and we know no
precise connection between our results and theirs.

In the next section we prove Theorem~\ref{main} and the above corollaries, and also give further applications of
Theorem~\ref{main} in which $h(x)\notin\F_Q[x]$.  In these latter applications, we choose $h$, $r$ and $m$ so that
$g(x)$ will have degree at most $5$.  We restrict to these degrees solely for the purpose of having a short list of
applications; our method can be used with higher-degree $g(x)$ to produce arbitrarily many further families of
permutation polynomials.  In Section~3 we use these new classes of permutation polynomials to construct new families of
complete permutation polynomials, and in particular we answer the open problem of Wu and Lin \cite{WL} by constructing
complete permutation polynomials over $\F_{2^{2^e}}$.
In Section~4 we use our new permutation polynomials to construct complete sets of mutually orthogonal latin
squares.  Quite special cases of some of our applications of Theorem~\ref{main} 
were obtained via lengthy computations in two recent
papers by Tu, Zeng and Hu \cite{TZH} and Xu, Cao, Tu, Zeng and Hu \cite{XCTZH}.
In Section~5 we explain how the results of \cite{TZH} and \cite{XCTZH}
follow from our results.


\section{Permutation polynomials from permutations of subfields}

In this section we exhibit some families of permutation polynomials
over $\F_{Q^m}$ which can be obtained from permutation polynomials over $\F_Q$, and in particular we prove
Theorem~\ref{main}.  Our constructions rely on the following result.

\begin{lemma} \label{lem}
Pick $h\in\F_q[x]$ and integers $r,s>0$ such that $s\mid (q-1)$.  Then $f(x):=x^r h(x^{(q-1)/s})$
permutes\/ $\F_q$ if and only if
\begin{enumerate}
\item $\gcd(r,(q-1)/s)=1$ \,\text{ and}
\item $x^r h(x)^{(q-1)/s}$ permutes the set of $s$-th roots of unity in\/ $\F_q^*$.
\end{enumerate}
\end{lemma}

This lemma has been used in several investigations of permutation polynomials, for instance see
\cite{MZ,TZ,Z1,Z2,Z3,Z4}.  Its short proof is given in \cite{TZ,Z1,Z2,Z4}.

We now deduce Theorem~\ref{main} from Lemma~\ref{lem}.

\begin{proof}[Proof of Theorem~\ref{main}]
In light of Lemma~\ref{lem}, it suffices to show that $g(x)$ and $\tilde g(x):=x^r h(x)^{(q-1)/(Q-1)}$ induce the
same mapping on $\F_Q^*$.  For $\beta\in\F_Q^*$, the fact that $(q-1)/(Q-1)=1+Q+Q^2+\dots+Q^{m-1}$ implies that
\[
\tilde g(\beta) = \beta^r h(\beta)^{(q-1)/(Q-1)} = \beta^r \prod_{j=0}^{m-1} h(\beta)^{Q^j}
= \beta^r \prod_{j=0}^{m-1} h^{(Q^j)}(\beta) = g(\beta),
\]
which completes the proof.
\end{proof}

Corollary~\ref{maincor} follows at once from Theorem~\ref{main}.  To deduce Corollary~\ref{maincor2} from
Corollary~\ref{maincor}, note that $x^n$ permutes $\F_Q$ (since $n$ is coprime to $Q-1$), so $g(x):=x^r h(x)^m$ permutes $\F_Q$
if and only if $g(x)^n=x^{rn} h(x)^{mn}$ permutes $\F_Q$; since $mn\equiv 1\pmod{Q-1}$, this last condition says that
$x^{rn} h(x)$ permutes $\F_Q$.

Corollaries~\ref{maincor} and \ref{maincor2} apply Theorem~\ref{main} in the special case that $h(x)\in\F_Q[x]$.
In the rest of this section we demonstrate how to apply Theorem~\ref{main} when $h(x)\notin\F_Q[x]$.
For simplicity, we restrict to the case that $g(x)$ is a member of
some of the simplest classes of permutation polynomials over $\F_Q$:  specifically, we use
permutation polynomials over $\F_Q$ of degree at most $5$.  These low-degree permutation polynomials were
classified in Dickson's thesis \cite[\S 87]{Di}.  We now recall a simplified version of Dickson's result.

\begin{lemma}\label{dickson}
The following polynomials permute\/ $\F_Q$:
\begin{enumerate}
\item $x^3$ \quad if $Q\not\equiv 1\pmod{3}$ \\
\item $x^3-\beta x$ \quad if $Q=3^n$ and $\beta\in\F_Q^*$ is a nonsquare \\
\item $x^4+\beta x^2+\gamma x$ \quad if $Q=2^n$ and $\beta,\gamma\in\F_Q$ and $x^4+\beta x^2+\gamma x$ has no nonzero roots in $\F_Q$ \\
\item $x^5$ \quad if $Q\not\equiv 1\pmod{5}$ \\
\item $x^5+\beta x^3+5^{-1}\beta^2x$ \quad if $Q\equiv\pm 2\pmod{5}$ and $\beta\in\F_Q$ \\
\item $x^5-\beta x$ \quad if $Q=5^n$ and $\beta\in\F_Q$ is not a fourth power \\
\item $x^5+2\beta x^3+\beta^2x$ \quad if $Q=5^n$ and $\beta\in\F_Q$ is a nonsquare.
\end{enumerate}
Conversely, for every degree-$3$ permutation polynomial $f(x)$ over\/ $\F_Q$, there exist $\theta,\mu,\nu\in\F_Q$ with $\theta\ne 0$ such that
$\theta f(x+\mu)+\nu$ is on the above list.  The same is true for degree-$4$ permutation polynomials if $Q\ne 2,3,7$, and for
degree-$5$ permutation polynomials if $Q>7$ and $Q\ne 13$.
\end{lemma}

We now use degree-$3$ permutation polynomials over $\F_Q$ to construct degree-$(Q+2)$ permutation polynomials over $\F_{Q^2}$:

\begin{cor}\label{cubic}
Pick $\alpha\in\F_{Q^2}^*$, and write $f_\alpha(x):=x^{Q+2}+\alpha x$.  Then $f_{\alpha}$ permutes\/ $\F_{Q^2}$ if and only if one of the following occurs:
\begin{enumerate}
\item $Q\equiv 5\pmod{6}$ and ${\alpha}^{Q-1}$ has order $6$;
   \item  $Q\equiv 2\pmod{6}$ 
     and ${\alpha}^{Q-1}$ has order $3$; or
\item $Q\equiv 0\pmod{3}$ and ${\alpha}^{Q-1}=-1$.
\end{enumerate}
In particular, the number of elements ${\alpha}\in\F_{Q^2}^*$ for which $f_{\alpha}$
permutes\/ $\F_{Q^2}$ is either $2(Q-1)$ or $Q-1$ or $0$, depending on
whether $Q$ is congruent to $2$, $0$ or $1$ modulo~$3$.
\end{cor}

This result is from \cite{TZ}.  Since that paper is unpublished, we include the proof.

\begin{proof}
By Theorem~\ref{main}, $f_{\alpha}$ permutes $\F_{Q^2}$ if and only if the polynomial $g_{\alpha}(x):=x(x+\alpha)(x+\alpha^Q)$
permutes $\F_Q$.  By Lemma~\ref{dickson}, the latter condition never occurs if $Q\equiv 1\pmod{3}$.

Now suppose that $Q\equiv 2\pmod{3}$, so that $g_{\alpha}$ permutes $\F_Q$ if and only if 
$g_{\alpha}(x)=(x+\mu)^3-\mu^3$ for some $\mu\in\F_Q$.  Note that $\mu\ne 0$ (since $\alpha\ne 0$).
We have $(x+\mu)^3-\mu^3=x(x+\mu-\omega\mu)(x+\mu-\omega^2\mu)$ where $\omega$ is a primitive cube root
of unity.  By unique factorization in $\F_{Q^2}[x]$, and the fact that $\omega\notin\F_Q$,
it follows that $f_{\alpha}$ permutes $\F_{Q^2}$ if and only if
$\alpha=\mu(1-\omega)$ for some primitive cube root of unity $\omega$ and some $\mu\in\F_Q^*$.
Equivalently, $\alpha^{Q-1}=(1-\omega)^{Q-1}$, which we compute to be
\[
(1-\omega)^{Q-1}=\frac{(1-\omega)^Q}{1-\omega}=\frac{1-\omega^Q}{1-\omega}  
   = \frac{1-\omega^2}{1-\omega}=1+\omega=-\omega^2.
\]
It follows that $(1-\omega)^{Q-1}$ has order six if $Q$ odd, and order three
if $Q$ even, and conversely each element of these orders occurs as $(1-\omega)^{Q-1}$ for some choice of $\omega$.
This concludes the proof when $Q\equiv 2\pmod{3}$.

Now suppose that $Q\equiv 0\pmod{3}$.  By Lemma~\ref{dickson}, the only cubic permutation polynomials over $\F_Q$
which are monic and divisible by $x$ are the polynomials $x^3-\beta x$ where $\beta\in\F_Q$ is either $0$ or a nonsquare.
For $\beta\ne 0$, any such polynomial is the product of $x$ and an irreducible quadratic in $\F_Q[x]$, so it equals
$x(x+\alpha)(x+\alpha^Q)$ if and only if $\alpha^2=\beta$.  It follows that, for $\alpha\in\F_{Q^2}^*$, the polynomial
$g_{\alpha}$ permutes $\F_Q$ if and only if $\alpha^2$ is a nonsquare in~$\F_Q$, or equivalently
$\alpha^{Q-1}=-1$.
\end{proof}

Likewise, using degree-$4$ permutation polynomials over $\F_r$ yields the following result from \cite{TZ}.

\begin{cor}
\label{quartic}
Let $Q$ be a prime power.  For $\alpha\in\F_{Q^3}^*$, the polynomial $x^{Q^2+Q+2}+{\alpha}x$
permutes\/ $\F_{Q^3}$ if and only if one of the following occurs:
\begin{enumerate}
\item $Q$ is even and $\alpha^{Q^2}+\alpha^Q+\alpha=0$;
\item $Q=7$ and $\alpha^{24}+\alpha^{18}+4\alpha^{12}+2=0$;
\item $Q=3$ and $\alpha^{12}+\alpha^{10}+\alpha^4+1=0$; or
\item $Q=2$ and $\alpha\ne 1$.
\end{enumerate}
The number of elements $\alpha\in\F_{Q^3}^*$ having the stated properties is $Q^2-1$, $24$, $12$ and $6$ in cases $(1)$--$(4)$.
For $\alpha\in\F_{Q^2}^*$, the polynomial $x^{Q+3}+\alpha x^2$ permutes\/ $\F_{Q^2}$ if and only if $Q=2$ and $\alpha\ne 1$.
\end{cor}

\begin{proof}
First assume $\alpha\in\F_{Q^3}^*$.
By Theorem~\ref{main}, $x^{Q^2+Q+2}+{\alpha}x$ permutes $\F_{Q^3}$ if and only if $g_{\alpha}(x):=x(x+\alpha)(x+\alpha^Q)(x+\alpha^{Q^2})$
permutes $\F_Q$.  By Lemma~\ref{dickson}, the latter condition never occurs if $Q$ is odd, except possibly when $Q\in\{3,7\}$.
Since it is easy to verify the result directly for $Q\le 7$, we assume henceforth that $Q$ is even and $Q>2$.  In this case,
$g_{\alpha}$ permutes $\F_Q$ if and only if $\alpha+\alpha^Q+\alpha^{Q^2}=0$ and $g_{\alpha}$ has no nonzero roots in
$\F_Q$, and one easily checks that the latter condition follows from the former since $\alpha\ne 0$.

Now assume $\alpha\in\F_{Q^2}^*$.  As above, $x^{Q+3}+\alpha x^2$ permutes $\F_{Q^2}$ if and only if
$\gcd(2,Q+1)=1$ and $g_{\alpha}(x):=x^2(x+\alpha)(x+\alpha^Q)$ permutes $\F_Q$.
The first condition says that $Q$ is even, so if $Q>2$ then Lemma~\ref{dickson} implies that $\alpha$ is not in $\F_Q$
and $g_{\alpha}$ has no degree-three term, which cannot both occur.  Finally, the result is clear when $Q=2$.
\end{proof}

Finally, using degree-$5$ permutation polynomials over $\F_Q$ yields the following result.

\begin{cor}
\label{quintic}
Pick any prime power $Q$, and any ${\alpha}\in\F_{Q^2}^*$.  The polynomial $x^{2Q+3}+{\alpha}x$ permutes\/ $\F_{Q^2}$
if and only if one of the following holds:
\begin{enumerate}
\item $Q\equiv\pm 2\pmod{5}$ and $\alpha^{2Q-2}-3\alpha^{Q-1}+1=0$;
\item $Q=5^n$ and either ${\alpha}^{Q-1}=-1$ or ${\alpha}^{(Q-1)/2}=-1$;
\item $Q=13$ and $\alpha^{12}-3\alpha^6+1=0$;
\item $Q=5$ and ${\alpha}^4 - {\alpha}^2+1=0$; or
\item $Q=3$ and either $\alpha=1$ or ${\alpha}^2=-1$.
\end{enumerate}
The number of elements $\alpha\in\F_{Q^2}$ having the stated properties is $2Q-2$ in case $(1)$, $3(Q-1)/2$ in case $(2)$,
and $12$, $4$ and $3$ in cases $(3)$, $(4)$ and $(5)$.
\end{cor}

\begin{proof}
By Theorem~\ref{main}, $x^{2Q+3}+\alpha x$ permutes $\F_{Q^2}$ if and only if $g_{\alpha}(x):=x(x^2+\alpha)(x^2+\alpha^Q)$
permutes $\F_Q$.  The result is easy to verify if $Q\le 13$, so we assume throughout that $Q>13$.

First suppose $Q\not\equiv 0\pmod{5}$.  By Lemma~\ref{dickson}, $g_{\alpha}$ permutes $\F_Q$ if and only if
$Q\equiv \pm 2\pmod{5}$ and $g_{\alpha}(x)=x^5+\beta x^3+5^{-1}\beta^2x$ for some $\beta\in\F_Q^*$, or equivalently
$(x+\alpha)(x+\alpha^Q)=x^2+\beta x+5^{-1}\beta^2$.  If this last equality holds then
\[
\alpha^{Q-1} + \alpha^{1-Q} = \frac{(\alpha^Q+\alpha)^2}{\alpha^{Q+1}} - 2 = \frac{\beta^2}{5^{-1}\beta^2} - 2 = 5 - 2 = 3,
\]
so that $\alpha^{2Q-2}-3\alpha^{Q-1}+1=0$.  Conversely, if $\alpha\in\F_{Q^2}^*$ satisfies
$\alpha^{2Q-2}-3\alpha^{Q-1}+1=0$, then $(\alpha^Q+\alpha)^2=5\alpha^{Q+1}$, so $\beta:=\alpha^Q+\alpha$
satisfies $(x+\alpha)(x+\alpha^Q)=x^2+\beta x + 5^{-1}\beta^2$.
Here $\beta\in\F_Q$, and further $\beta\ne 0$ since otherwise we would have
$\alpha^{Q-1}=1$ so that $\alpha^{2Q-2}-3\alpha^{Q-1}+1=-1$ is nonzero.
In case $Q\not\equiv 0\pmod{5}$, it remains only to show that the polynomial $x^{2Q-2}-3x^{Q-1}+1$ has
$2Q-2$ roots in $\F_{Q^2}^*$.  It suffices to show that $x^2-3x+1$ is irreducible over $\F_Q$, since then the roots of
this polynomial in $\F_{Q^2}$ are $\gamma$ and $\gamma^Q$; since the product of the roots is $1$, it follows that
$\gamma^{Q+1}=1$, so that $\gamma$ has $(Q-1)$ distinct $(Q-1)$-th roots in $\F_{Q^2}^*$.  Thus, we need only
show that $x^2-3x+1$ is irreducible over $\F_Q$.
To this end, write $Q=p^j$ with $p$ prime; since $Q\equiv\pm 2\pmod{5}$, we see that
$j$ is odd and $p\equiv\pm 2\pmod{5}$.  Since $j$ is odd, irreducibility of $x^2-3x+1$ over $\F_Q$ is equivalent to
irreducibility over $\F_p$.  The latter irreducibility is clear if $p=2$, so assume $p$ is odd.  Then $x^2-3x+1$ is irreducible over
$\F_p$ if and only if its discriminant (namely $5$) is a nonsquare in $\F_p$, which by quadratic reciprocity is the same as requiring that $p$ is a nonsquare in $\F_5$, which indeed is the case since $p\equiv\pm 2\pmod{5}$.

Now suppose $Q\equiv 0\pmod{5}$.  Note that if $g_{\alpha}(x)$ has a term of degree $3$, then $g_{\alpha}(x+\mu)$
has a term of degree $2$ for any $\mu\in\F_Q^*$.  Thus, Lemma~\ref{dickson} implies that $g_{\alpha}$ permutes $\F_Q$
if and only if either
\begin{enumerate}
\item $\alpha+\alpha^Q=0$ and $-\alpha^{Q+1}$ is not a fourth power in $\F_Q$, or
\item $(x+\alpha)(x+\alpha^Q)=x^2+2\beta x+\beta^2$ for some nonsquare $\beta\in\F_Q$.
\end{enumerate}
In the first case, since $\alpha^{Q-1}=-1$, we have $\alpha\notin\F_Q$ and $-\alpha^{Q+1}=\alpha^2$, so
$-\alpha^{Q+1}$ is a nonsquare in $\F_Q$ and hence is not a fourth power.  In the second case,
the equality $(x+\alpha)(x+\alpha^Q)=(x+\beta)^2$ occurs just when $\alpha=\beta$.
\end{proof}


\section{Complete permutation polynomials}

In this section we use the results of the previous section to construct complete permutation polynomials.

\begin{cor}\label{lots}
Pick $\alpha\in\F_Q^*$, and put $q=Q^m$ where $m$ and $s$ are positive integers with $\gcd(m,Q-1)=1$.
Then $f(x):=\alpha x^{1+s(q-1)/(Q-1)}$ is a complete permutation polynomial over\/ $\F_q$ if and only if
\begin{enumerate}
\item $\gcd(1+s(q-1)/(Q-1),Q-1)=1$ \,\text{ and}
\item $\alpha x^{ms+1}+x$ permutes\/ $\F_Q$.
\end{enumerate}
\end{cor}

\begin{proof}
By Corollary~\ref{maincor2} with $r=1$ and $h(x)=\alpha x^s+1$, we see that $f(x)+x$ permutes $\F_q$ if and only if $g(x):=x^n (\alpha x^s+1)$ 
permutes $\F_Q$, where $mn\equiv 1\pmod{Q-1}$.
Since $\gcd(m,Q-1)=1$, the polynomial $x^m$ permutes $\F_Q$, so $g(x)$ permutes $\F_Q$ if and only if
$g(x^m)=x^{mn} (\alpha x^{ms}+1)$ permutes $\F_Q$.  Since $mn\equiv 1\pmod{Q-1}$, this last condition says that
$x(\alpha x^{ms}+1)$ permutes $\F_Q$.
  Next, $f(x)$ permutes $\F_q$ if and only if $1+s(q-1)/(Q-1)$ is coprime to $q-1$; since this number is clearly coprime to
$(q-1)/(Q-1)$, it suffices to ensure that it is coprime to $Q-1$.
\end{proof}

One can use Corollary~\ref{lots} to exhibit many families of complete permutation polynomials, by using the various known families
of permutation binomials over $\F_Q$.  We give only one instance of this, which suffices to answer the open problem in \cite{WL}.

\begin{cor} If $Q$ is a power of $4$, and $\alpha\in\F_Q$ is not a cube, then $\alpha x^{(Q+1)Q/2}$ is a complete permutation
polynomial over\/ $\F_{Q^2}$.
\end{cor}

\begin{proof}
We apply Corollary~\ref{lots} with $m=2$ and $s=Q/2+1$.  Since $1+s(Q+1)=(Q+1)Q/2$ is coprime to $Q-1$, it follows that
$\alpha x^{(Q+1)Q/2}$ is a complete permutation polynomial over $\F_{Q^2}$ if and only if $\alpha x^{Q+3}+x$ permutes $\F_Q$.
This polynomial induces the same function on $\F_Q$ as does $\alpha x^4+x$.  In particular, this function is a homomorphism from
the additive group of $\F_Q$ to itself, so it is bijective if and only if its kernel is trivial, which is the case since $\alpha$ is not a cube
in $\F_Q$.
\end{proof}

\begin{remark} The open problem in \cite{WL} asked whether there exist complete permutation polynomials over $\F_{2^{2^e}}$.
The previous corollary provides complete permutation polynomials more generally over every field $\F_{2^{4d}}$.  As noted above,
one can produce many further complete permutation polynomials as consequences of Corollary~\ref{lots}.
\end{remark}

The rest of the results in this section use the permutation polynomials from Corollaries~\ref{cubic}, \ref{quartic} and \ref{quintic} to
produce complete permutation polynomials.

\begin{cor}\label{cubic2}
For $\alpha\in\F_{Q^2}^*$ and $\beta\in\F_Q$, the polynomial $f(x):=\alpha x^{Q+2}+\beta x$ is a complete
permutation polynomial over\/ $\F_{Q^2}$ if and only if one of the following holds:
\begin{enumerate}
\item $Q\equiv 5\pmod{6}$ and ${\alpha}^{Q-1}$ has order $6$;
   \item  $Q\equiv 2\pmod{6}$ and ${\alpha}^{Q-1}$ has order $3$; or
\item $Q\equiv 0\pmod{3}$ and ${\alpha}^{Q-1}=-1$.
\end{enumerate}
\end{cor}

\begin{proof}
At least one polynomial in $\{f(x),f(x)+x\}$ has terms of degrees $Q+2$ and $1$, so by Corollary~\ref{cubic}
if $f(x)$ is a complete permutation polynomial then $Q\not\equiv 1\pmod{3}$.  Conversely, suppose that 
$Q\not\equiv 1\pmod{3}$.  
It follows that $\alpha x^{Q+2}$ permutes $\F_{Q^2}$: for, this is equivalent to requiring
$\gcd(Q+2,Q^2-1)=1$, and $\gcd(Q+2,Q^2-1)$ divides $\gcd(Q^2-4,Q^2-1)=\gcd(3,Q-1)=1$.
Each of $f(x)$ and $f(x)+x$ has the form $\alpha x^{Q+2}+\gamma x$ with $\gamma\in\F_Q$, and
at least one of them has $\gamma\in\F_Q^*$.  Since $\alpha x^{Q+2}+\gamma x$ permutes $\F_{Q^2}$
if and only if $x^{Q+2}+\gamma\alpha^{-1} x$ permutes $\F_{Q^2}$, and $(\gamma\alpha^{-1})^{Q-1}=
\alpha^{1-Q}$, the result now follows from Corollary~\ref{cubic}.
\end{proof}

\begin{cor}\label{quartic2}
For $\alpha\in\F_{Q^3}^*$ and $\beta\in\F_Q$, the polynomial $f(x):=\alpha x^{Q^2+Q+2}+\beta x$
is a complete permutation polynomial over\/ $\F_{Q^3}$ if and only if one of the following holds:
\begin{enumerate}
\item $Q$ is even and $\alpha^{Q^2}+\alpha^{Q^2-Q+1}+\alpha=0$;
\item $Q=7$ and $2\alpha^{24}+4\alpha^{12}+\alpha^6+1=0$ and $\beta\notin\{0,-1\}$;
\item $Q=3$ and $\alpha^{12}+\alpha^{8}+\alpha^2+1=0$ and $\beta=1$; or
\item $Q=2$ and $\alpha\ne 1$.
\end{enumerate}
\end{cor}

\begin{proof}
This follows from Corollary~\ref{quartic} in the same way that Corollary~\ref{cubic2} followed from Corollary~\ref{cubic}.
All that is required is to determine when $\gcd(Q^2+Q+2,Q^3-1)=1$.
Note that this gcd is even when $Q$ is odd, so we may assume that $Q$ is even.
Now $\gcd(Q^2+Q+2,Q^3-1)$ divides $(Q^2+Q+2)(Q-1)-(Q^3-1)=Q-1$ and hence divides
$Q^2+Q+2-(Q-1)(Q+2)=4$, so $\gcd(Q^2+Q+2,Q^3-1)=1$ when $Q$ is even.
\end{proof}

\begin{cor}\label{quintic2}
For ${\alpha}\in\F_{Q^2}^*$ and $\beta\in\F_Q$, the polynomial $\alpha x^{2Q+3}+\beta x$ is a complete permutation polynomial over\/ $\F_{Q^2}$
if and only if one of the following holds:
\begin{enumerate}
\item $Q\equiv\pm 2\pmod{5}$ and $\alpha^{2Q-2}-3\alpha^{Q-1}+1=0$;
\item $Q=5^n$ and ${\alpha}^{Q-1}=-1$;
\item $Q=5^n$ and  $\beta^{(Q-1)/2},(\beta+1)^{(Q-1)/2}\in\{0,-\alpha^{(Q-1)/2}\}$;
\item $Q=13$ and ${\alpha}^{12}-3{\alpha}^6+1=0$ and $\beta\in\{0,3,-4,-1\}$;
\item $Q=13$ and ${\alpha}^{12}+3{\alpha}^6+1=0$ and $\beta\in\{5,6,7\}$;
\item $Q=5$ and ${\alpha}^4 - {\alpha}^2+1=0$ and $\beta\in\{0,-1\}$;
\item $Q=5$ and ${\alpha}^4 + {\alpha}^2+1=0$ and $\beta=2$;
\item $Q=3$ and ${\alpha}^2=-1$; or
\item $Q=3$ and $\alpha+\beta=1$.
\end{enumerate}
\end{cor}

\begin{proof}
This follows from Corollary~\ref{quintic} using the same argument as in the proof of Corollary~\ref{cubic2}.
Again, it suffices to determine when $\gcd(2Q+3,Q^2-1)=1$.
Note that $\gcd(2Q+3,Q^2-1)$ equals $\gcd(2Q+3,Q+1)\cdot\gcd(2Q+3,Q-1)$, and since
$2Q+3=2(Q+1)+1$ and $2Q+3=2(Q-1)+5$ we see that $\gcd(2Q+3,Q^2-1)=1$ if and only if $Q\not\equiv 1\pmod{5}$.
\end{proof}

\begin{remark}
Many further families of complete permutation polynomials can be constructed using the same methods as above.
For instance, since the inverse of a complete permutation polynomial is again a complete permutation polynomial, one
can use the inverses of the complete permutation polynomials constructed above.  These inverses take an especially simple
form in case the polynomial itself is a monomial.
\end{remark}


\section{Mutually orthogonal latin squares}

In this section we explain how the permutation polynomials constructed in this paper can be used to produce
complete sets of mutually orthogonal latin squares, and consequently to produce projective planes.

We begin by recalling the definitions.  A \emph{latin square} of order $n$ is an $n\times n$ matrix with entries from an
$n$-element set $S$, such that each element of $S$ occurs exactly once in each row and each column.  Two such
squares $L_1=[\alpha_{ij}]$ and $L_2=[\beta_{ij}]$ which have the same set $S$ are called \emph{orthogonal} if
every ordered pair in $S\times S$ occurs as $(\alpha_{ij},\beta_{ij})$ for some $i,j$.  A set of pairwise orthogonal
latin squares is called a set of mutually orthogonal latin squares (MOLS).  Any set of MOLS of order $n$ has cardinality
at most $n-1$, and a set of $n-1$ MOLS of order $n$ is called a \emph{complete} set of MOLS of order $n$.
A classical argument of Bose \cite{Bose} shows that a complete set of MOLS of order $n$ can be used to produce a
projective plane of order $n$ (see also \cite[Thm.~5.2.2]{DK}).

We will always take $S$ to be the finite
field $\F_q$, and we will label the rows and columns of our latin squares by the elements of $\F_q$.  In this case, a permutation polynomial $f$ over $\F_q$ corresponds to the latin square whose $ij$-th
entry is $i+f(j)$, and the latin squares corresponding to two permutation polynomials $f,g$ are orthogonal if and only if
$f-g$ is a permutation polynomial.  Via this correspondence, the following consequences of Corollaries~\ref{cubic}, \ref{quartic}
 and \ref{quintic} exhibit complete sets of mutually orthogonal
latin squares of order $q$.

\begin{cor} Let $Q$ be a prime power with $Q\not\equiv 1\pmod{3}$, and let $\alpha$ be any fixed element of\/ $\F_{Q^2}$
for which $\alpha^{Q-1}$ has order either $6$ (if $Q\equiv 5\pmod{6}$) or $3$ (if $Q\equiv 0\pmod{2}$) or
$2$ (if $Q\equiv 0\pmod{3}$).  Then the set of all polynomials $\beta x^{Q+2}+\alpha\gamma x$ with $\beta,\gamma\in\F_Q$,
where at least one of $\beta,\gamma$ is nonzero, corresponds to a complete set of $Q^2-1$ MOLS of order $Q^2$.
\end{cor}

\begin{cor} Let $Q$ be a power of $2$.  Then the set of all polynomials $\beta x^{Q^2+Q+2}+\alpha x$ with $\beta\in\F_Q$
and $\alpha^{Q^2}+\alpha^Q+\alpha=0$, where at least one of $\alpha,\beta$ is nonzero, corresponds to a complete
set of $Q^3-1$ MOLS of order $Q^3$.
\end{cor}

\begin{cor} Let $Q$ be a prime power which is congruent mod $5$ to either $0$, $2$, or $-2$,
and let $\alpha$ be any fixed element of\/ $\F_{Q^2}$
for which $\alpha^{2Q-2}-3\alpha^{Q-1}+1=0$.  Then the set of all polynomials $\beta x^{2Q+3}+\alpha\gamma x$
with $\beta,\gamma\in\F_Q$, where at least one of $\beta,\gamma$ is nonzero, corresponds to a complete set of
$Q^2-1$ MOLS of order $Q^2$.
\end{cor}


\section{The results of Xu, Cao, Tu, Zeng and Hu}

In this section we show how our results imply all three main results in \cite{XCTZH}, as
well as three of the four main results in \cite{TZH}.  We note that the proofs in \cite{TZH} and \cite{XCTZH}
involve lengthy computations based on a completely different method than the one in the present note.

The first main result of \cite{XCTZH} is \cite[Thm.~3.1]{XCTZH}, which follows from the special case of
Corollary~\ref{cubic2} in which $\beta=0$ and $Q=3^j$ with $j$ odd.  Note that Corollary~\ref{cubic2} exhibits complete
permutation polynomials for every prime power $Q$ such that $Q\not\equiv 1\pmod{3}$.

Likewise, \cite[Thm.~1]{TZH} follows from the special case of Corollary~\ref{quartic2} in which $\beta=0$ and $Q=2^j$ with
$\gcd(j,3)=1$.  Note that Corollary~\ref{quartic2} exhibits complete permutation polynomials whenever $Q$ is a power of $2$.

Next, \cite[Thm.~2]{XCTZH} and \cite[Thm.~3.3]{XCTZH} follow from the special cases of Corollary~\ref{quintic2} in which
$\beta=0$ and either $Q=2^j$ or $Q=3^j$.  Note that Corollary~\ref{quintic} exhibits complete permutation polynomials for every prime power $Q$
which is congruent to $0$, $2$ or $3$ (mod~$5$).
%

The third main result of \cite{XCTZH} is an immediate consequence of the following result from our previous paper \cite[Cor.~5.3]{Z4}
(which itself is a simple consequence of Lemma~\ref{lem}):

\begin{lemma} \label{genlem} Let $Q$ be a prime power, let $r$ and $d$ be positive integers, and let $\beta$ be a
$(Q+1)$-th root of unity in\/ $\F_{Q^2}$.  Then $x^{r+d(Q-1)}+\beta^{-1} x^r$ permutes\/ $\F_{Q^2}$
if and only if all of the following hold:
\begin{enumerate}
\item $\gcd(r,Q-1)=1$
\item $\gcd(r-d,Q+1)=1$
\item $(-\beta)^{(Q+1)/\gcd(Q+1,d)}\ne 1$.
\end{enumerate}
\end{lemma}

By restricting to the case $r=1$, we obtain the following complete permutation polynomials from this result:

\begin{cor} Let $Q$ be a prime power, let $d$ be a positive integer, and let $\beta$ be a $(Q+1)$-th root of unity in\/ $\F_{Q^2}$.
Then $\beta x^{1+d(Q-1)}$ is a complete permutation polynomial over\/ $\F_{Q^2}$ if and only if all of the following hold:
\begin{enumerate}
\item $\gcd((d-1)(2d-1),Q+1)=1$
\item $(-\beta)^{(Q+1)/\gcd(Q+1,d)}\ne 1$.
\end{enumerate}
\end{cor}

By restricting to the special case that $2d\mid (Q+1)$, we obtain a stronger version of \cite[Thm.~3.5]{XCTZH}.
By restricting to the special case that $Q$ is even and $d=Q/4+1$, we obtain a stronger version of \cite[Thm.~3]{TZH}.

The results in this paper do not imply \cite[Thm.~4]{TZH}.  That result was proved by using a particular case of a generalization
of Lemma~\ref{lem} from \cite{AGW}.  Whereas the proof of Lemma~\ref{lem}  involves the multiplicative group of $\F_q$, the proof of \cite[Thm.~4]{TZH} involves instead the additive group of $\F_q$.  It would be good to understand the general class of
all permutation polynomials which can be produced by variants of the method used to prove \cite[Thm.~4]{TZH}.



\begin{thebibliography}{1}

\bibitem{AGW} A. Akbary, D. Ghioca and Q. Wang, On constructing permutations of finite fields,
\textit{Finite Fields Appl.} \textbf{17} (2011), 51--67.

\bibitem{Bose} R. C. Bose, On the application of the properties of Galois fields to the construction of
hyper-Graeco-Latin squares, \textit{Sankhy{\={a}}} \textbf{3} (1938), 323--338.

\bibitem{CH} L. Carlitz and D. R. Hayes, Permutations with coefficients in a subfield,
\textit{Acta Arith.} \textbf{21} (1972), 131--135.

\bibitem{DK} J. D\'enes and A. D. Keedwell, Latin Squares and their Applications,
Academic Press, New York, 1974.

\bibitem{Di} L. E. Dickson, The analytic representation of substitutions on a power of a prime number of letters with a discussion
of the linear group, \textit{Ann. of Math.} \textbf{11} (1896/97), 65--120.

\bibitem{FGM} R. J. Friedlander, B. Gordon and M. D. Miller, On a group sequencing problem of Ringel,
in: Proceedings of the Ninth Southeastern Conference on Combinatorics, Graph Theory, and Computing, 307--321,
Utilitas Math., Winnipeg, 1978.

\bibitem{FGT} R. J. Friedlander, B. Gordon and P. Tannenbaum, Partitions of groups and complete mappings,
\textit{Pacific J. Math.} \textbf{92} (1981), 283--293.

\bibitem{Mann} H. B. Mann, The construction of orthogonal latin squares, \textit{Ann. Math. Statist.} \textbf{13} (1942), 418--423.

\bibitem{MZ}
A. M. Masuda and M. E. Zieve, Permutation binomials over finite fields,
\textit{Trans. Amer. Math. Soc.} \textbf{361} (2009), 4169--4180.

\bibitem{NR}
H. Niederreiter and K. H. Robinson, Bol loops of order $pq$, \textit{Math. Proc. Cambride Philos. Soc.} \textbf{89} (1981), 241--256.

\bibitem{Paige}
L. J. Paige, Neofields, \textit{Duke Math. J.} \textbf{16} (1949), 39--60.

\bibitem{TZH}
Z. Tu, X. Zeng and L. Hu, Several classes of complete permutation polynomials,
\textit{Finite Fields Appl.} \textbf{25} (2014), 182--193.

\bibitem{TZ}
T. J. Tucker and M. E. Zieve, Permutation polynomials, curves without points, and Latin squares,
preprint, 2000.

\bibitem{WL}
B. Wu and D. Lin, On constructing complete permutation polynomials over finite fields of even characteristic,
arXiv:1310.4358v2 [math.NT], 29 Oct 2013.

\bibitem{XCTZH}
G. Xu, X. Cao, Z. Tu, X. Zeng and L. Hu, Complete permutation polynomials over finite fields of odd characteristic,
arXiv:1312.0930v1 [math.NT], 1 Dec 2013.

\bibitem{Z1}
M. E. Zieve, Some families of permutation polynomials over finite fields,
\textit{Internat. J. Number Theory} \textbf{4} (2008), 851--857.

\bibitem{Z2}
M. E. Zieve, On some permutation polynomials over $\F_q$ of the form $x^r h(x^{(q-1)/d})$,
\textit{Proc. Amer. Math. Soc.}  \textbf{137} (2009), 2209--2216.

\bibitem{Z3}
M. E. Zieve, Classes of permutation polynomials based on cyclotomy and an additive analogue,
in \textit{Additive Number Theory}, Springer (2010), 355--359.

\bibitem{Z4}
M. E. Zieve, Permutation polynomials on $\F_q$ induced from R\'edei function bijections on subgroups of $\F_q^*$,
arXiv:1310.0776v2 [math.NT], 7 Oct 2013.

\end{thebibliography}
\end{document}